\documentclass[12pt]{article}
\usepackage{amsmath,amssymb,amsthm}
\usepackage{amsfonts, amsmath}
\newtheorem{theorem}{Theorem}[section]
\newtheorem{lemma}{Lemma}[section]

\theoremstyle{definition}

\theoremstyle{remark}

\numberwithin{equation}{section}

 \title{On the construction of fully interpreted formal languages which posses their truth predicates} 
 
\author {S. Heikkil\"a\\
Department of Mathematical Sciences, University of Oulu\\
BOX 3000, FIN-90014, Oulu, Finland\\
E-mail: sheikki@cc.oulu.fi}
\begin{document}
\maketitle 
Shorttitle: Languages which posses their truth predicates

\noindent
\begin{abstract} 
\noindent
We shall first construct by ordinary recursion method subsets to the set $D$ of G\"odel numbers of the sentences of a language $\mathcal L$. That language is formed by the sentences of a fully interpreted formal language $L$, called an MA language, and sentences containing a monadic predicate letter $T$. From the class of the constructed subsets of $D$ we extract one set $U$ by transfinite recursion method.
 Interpret those sentences whose G\"odel numbers are in $U$ as true, and their negations as false. These sentences together  form an
MA language. It is a sublanguage of $\mathcal L$ having $L$ as its sublanguage, and $T$ is its truth predicate.
 
\vskip12pt

\vskip12pt


\end{abstract}
\newpage

\baselineskip 15pt

\section {Introduction}\label{S1} 

In  \cite{Hei14} a theory of truth is defined for certain sublanguages of a language $\mathcal L$ which is the first order language $L=\{\in\}$ of set theory augmented by a monadic predicate $T$.  
The interpretation of $L$ is determined by the  minimal model $M$ constructed in \cite{[4]} for ZF set theory.
This interpretation makes $L$ fully interpreted, i.e., its sentences are either true or false.
The sublanguages for which a theory of truth is defined belong to a class of sublanguages  of $\mathcal L$. Languages of that class are denoted by $\mathcal L_U$, where $U$
is a subset of the set $D$ of the G\"odel numbers of sentences of $\mathcal L$.  The  G\"odel numbers of sentences of $\mathcal L_U$ belong to the set $G(U)\cup F(U)$, where the subsets  $G(U)$ and $F(U)$ of $D$ are told to satisfy the following rules ('iff' abbreviates 'if and only if'):  
 \begin{enumerate}
 \item[(r1)] If $A$ is a sentence of $L$, then the G\"odel number \#$A$ of $A$ is in  $G(U)$ iff  $A$ is true in the interpretation of $L$, and  in  $F(U)$ iff $A$ is false in the interpretation of $L$.  
 \item[(r2)] Let $\mathbf n$ be a numeral. \#$T(\mathbf n)$ is in $G(U)$ iff $\mathbf n$ is the numeral $\left\lceil A\right\rceil$ of the G\"odel number of a sentence $A$ of $\mathcal L$ and  \#$A$  is in $U$. \#$T(\mathbf n)$ is in $F(U)$ iff $\mathbf n=\left\lceil A\right\rceil$, where $A$ is a sentence of $\mathcal L$ and \#[$\neg A$] is in $U$. 
\end{enumerate}  

In the next rules (r3)--(r7) $A$ and $B$ denote sentences of $\mathcal L$. 
\begin{enumerate}
 \item[(r3)] Negation rule: \#[$\neg A$] is in $G(U)$ iff \#$A$ is in $F(U)$, and in $F(U)$ iff \#$A$ is in $G(U)$. 
 \item[(r4)] Disjunction rule: \#[$A\vee B$]
 is  in $G(U)$ iff  \#$A$ or \#$B$ is in $G(U)$, and in $F(U)$ iff \#$A$ and \#$B$
are in $F(U)$. 
 \item[(r5)] Conjunction rule: \#[$A\wedge B$]
is in $G(U)$ iff 
both \#$A$ and \#$B$ are in $G(U)$. \#[$A\wedge B$] is in $F(U)$ iff 
\#$A$ or \#$B$ is in $F(U)$.
\item[(r6)] Implication rule: \#[$A\rightarrow B$] is in $G(U)$ iff 
\#$A$ is in  $F(U)$ or \#$B$ is in $G(U)$.  \#[$A\rightarrow B$] is  in $F(U)$ iff 
\#$A$ is in $G(U)$ and \#$B$ is in $F(U)$.
\item[(r7)] Biconditionality rule:
\#[$A \leftrightarrow B$] is in $G(U)$ iff \#$A$ and \#$B$ are both in $G(U)$ or both in $F(U)$, and  in $F(U)$ iff  \#$A$ is in $G(U)$ and 
\#$B$ is in $F(U)$ or \#$A$ is in $F(U)$ and \#$B$ is in $G(U)$.
\end{enumerate}
Assuming that the set $X$ of numerals of G\"odel numbers of sentences of $\mathcal L$ is the intended domain of discourse for $T$, the following rules are presented for $\exists xT(x)$ and $\forall xT(x)$: 
\begin{itemize}
\item[(r8)] \#[$\exists xT(x)$] is in $G(U)$ iff \#$T(\mathbf n)$ is in $G(U)$ for some $\mathbf n\in X$, and \#[$\exists xT(x)$] is in $F(U)$ iff  \#$T(\mathbf n)$ is in $F(U)$ for every $\mathbf n\in X$.
\item[(r9)] \#[$\forall xT(x)$] is in $G(U)$ iff \#$T(\mathbf n)$ is in $G(U)$ for every  $\mathbf n\in X$, and \#[$\forall xT(x)$] is in $F(U)$ iff \#$T(\mathbf n)$ is in $F(U)$ at least for one $\mathbf n\in X$.
\end{itemize}

In \cite{Hei15} the above considerations are extended to the case when the language $L$ is assumed to be mathematically agreeable (shortly MA). By  Chomsky's definition (cf. \cite{[C]}) a ``language is a set (finite or infinite) of sentences  of finite length, and constructed out of finite sets of symbols". Allowing also countable sets of symbols, we say that $L$ is 
an MA language if it satisfies the following conditions.
\smallskip

(i) The syntax of  $L$ contains a countable syntax of the first-order predicate logic with equality 
(cf., e.g.,  \cite[II.5]{[Ku]}), 
natural numbers in variables and their names, numerals in terms.
 
(ii) $L$ is fully interpreted.  

(iii) Classical truth tables (cf. e.g., \cite[p.3]{[Ku]}) are valid for the logical connectives 
$\neg$, $\vee$, $\wedge$, $\rightarrow$ and $\leftrightarrow$ of sentences of $L$.

(iv) Classical rules of truth hold for $\forall xP(x)$ and $\exists xP(x)$ where $P$ is a predicate of $L$.
\smallskip

 Any countable first-order formal language, equipped with a consistent theory interpreted by a countable model, and containing natural numbers and numerals, is an MA language. A classical example is the language of arithmetic  with its standard model and interpretation.

Basic ingredients in the  approach of \cite{Hei15} are: 

1. An MA language $L$ (base language). 

2. A monadic predicate letter $T$.

3. The language $\mathcal L$, which has  sentences of $L$, $T(\mathbf n)$, where $\mathbf n$ is a numeral, $\forall xT(x)$ and $\exists xT(x)$ as its basic sentences, and which is closed under  logical connectives $\neg$, $\vee$, $\wedge$, $\rightarrow$ and $\leftrightarrow$.

4. The set $D$ of G\"odel numbers of sentences of $\mathcal L$ in its fixed G\"odel numbering. 
\smallskip

Neither in \cite{Hei14} nor in \cite{Hei15} the sets $G(U)$ and $F(U)$  satisfying rules (r1)--(r9) are shown to exist. Our main task is to construct sets  $G(U)$ and $F(U)$, and prove that they satisfy rules (r1)--(r9), and the following rules when \#[$\exists x\neg T(x)$] and  \#[$\forall x\neg T(x)$] are added to basic sentences of $\mathcal L$.
\begin{itemize}
\item[(r10)] \#[$\exists x\neg T(x)$] is in $G(U)$ iff \#$T(\mathbf n)$ is in $F(U)$ for some $\mathbf n\in X$, and \#[$\exists x\neg T(x)$] is in $F(U)$ iff  \#$T(\mathbf n)$ is in $G(U)$ for every $\mathbf n\in X$.
\item[(r11)] \#[$\forall x\neg T(x)$] is in $G(U)$ iff \#$T(\mathbf n)$ is in $F(U)$ for every  $\mathbf n\in X$, and \#[$\forall x\neg T(x)$] is in $F(U)$ iff \#$T(\mathbf n)$ is in $G(U)$ at least for one $\mathbf n\in X$.
\end{itemize}

Because of the recursive construction of sets $G(U)$ we revise proofs given in \cite{Hei14, Hei15} for  properties of sets $G(U)$ and $F(U)$. 
Some of them are used in \cite[Theorem 4.1]{Hei14} to prove by transfinite recursion method the existence of consistent fixed points $U$ of $G$, i.e., subsets $U$ of $D$ which satisfy  $U=G(U)$, and  for no sentence $A$ of $\mathcal L$ the G\"odel numbers of both $A$ and $\neg A$ are in $U$. Among them there is the smallest one which is contained in every consistent fixed point of $G$.
\smallskip

To the smallest consistent fixed point $U$ of $G$ there corresponds a sublanguage $\mathcal L_0$ of $\mathcal L$ which has $G(U)\cup F(U)$ as the set of G\"odel numbers of its sentences.
As in \cite{Hei14,Hei15}, define an interpretation  for sentences of $\mathcal L_0$ as follows.
A sentence $A$ of $\mathcal L_0$ is interpreted as true iff its G\"odel number \#$A$ is in $G(U)$, and as false iff \#$A$ is in $F(U)$. 
The so defined theory of truth for $\mathcal L_0$ is shown in \cite[Theorem 3.1]{Hei14} to conform well with
the  eight norms presented for theories of truth in \cite{Lei07}. $T$ is called a truth predicate for $\mathcal L_0$, because 
$T$-biconditionality: 
$A\leftrightarrow T(\left\lceil A\right\rceil)$, is shown in  \cite[Lemma 4.1]{Hei15} to be true in that interpretation for all sentences $A$ of $\mathcal L_0$ ($T(\left\lceil A\right\rceil)$ stands for '$A$ is true'). 
 $\mathcal L_0$ is fully interpreted because $G(U]$ and $F(U)$ are separated by Lemma \ref{L01}. 

$L$ is by rule (r1) a sublanguage of $\mathcal L_0$. Moreover, 
a sentence $A$ of $L$ is by \cite[Lemma 4.2]{Hei15} true, respectively false, in the interpretation of $L$ if and only if $A$ 
is true, respectively false, in the interpretation of $\mathcal L_0$.  

$\mathcal L_0$ is an MA language. By above results it satisfies properties (i) and (ii).  
Rules (r1)--(r11) and the above defined interpretation imply that properties (iii) and
(iv) are valid by assuming that  $T$ and $\neg T$ are the only predicates of $\mathcal L_0$ which are not predicates of $L$.

 
Main tools used in proofs are ZF set theory and classical logic.  
 
 
\section{Recursive construction of sets $G(U)$}\label{S2} 

Let  $L$, $T$,  $\mathcal L$ and $D$ be 
as in the Introduction, and let $W$ denote the set of G\"odel numbers of all those sentences of $L$ which are true in the interpretation of $L$.
Given a subset $U$ of  $D$, 
denote 
\begin{equation}\label{E20}
\begin{aligned}
D_1(U)=&\{\hbox{\#$T(\mathbf n)$: $\mathbf n=\left\lceil A\right\rceil$, where $A$ is a sentence of $\mathcal L$ and  \#$A$  is in $U$}\},\\
D_2(U)=&\{\hbox{\#[$\neg T(\mathbf n)$]: $\mathbf n=\left\lceil A\right\rceil$, where $A$ is a sentence of $\mathcal L$ and  \#[$\neg A$]  is in $U$}\}.
\end{aligned}
\end{equation}
We shall  construct subsets  $G_n(U)$, $n\in\mathbb N_0$, of $D$
 recursively as follows.

Define  
\begin{equation}\label{E201}
G_0(U)=\begin{cases} W \ \hbox {if $U=\emptyset$},\\
  W\cup D_1(U)\cup \{\#[\exists xT(x)],\#[\neg\forall x\neg T(x)]\} 
	\hbox{ if $\emptyset\subset U\subset D$, and $D_2(U)=\emptyset$,}\\
  W\cup D_1(U)\cup D_2(U)\cup \{\#[\exists xT(x)],\#[\neg\forall x\neg T(x)]\}\\
	\cup\{\#[\neg\forall xT(x)]\#[\exists x\neg T(x)]\} \
	\hbox{ if $\emptyset\subset U\subset D$, and $D_2(U)\ne\emptyset$}, \\               
										W\cup D_1(U)\cup D_2(U)\cup \{\#[\exists xT(x)],\#[\neg\exists xT(x)]
																				,\#[\forall xT(x)],\#[\neg\forall xT(x)]\} \\
											\cup\{\#[\forall x\neg T(x)],\#[\neg\forall x\neg T(x)],
											\#[\exists x\neg T(x)]\#[\neg\exists x\neg T(x)]\hbox{ if } U=D.
										\end{cases}
\end{equation}										

Let  $A$ and $B$ denote sentences of $\mathcal L$. When $n\in\mathbb N_0$, and $G_n(U)$ is defined, denote
\begin{equation}\label{E203}
\begin{aligned}
G_n^1(U)=&\{\#[A\vee B]:\#A \hbox{ or \#$B$ is in } G_n(U)\},\\
G_n^2(U)=&\{\#[A\wedge B]:\#A \hbox{ and \#$B$ are in } G_n(U)\},\\
G_n^3(U)=&\{\#[A\rightarrow B]:\#[\neg A] \hbox{ or \#$B$ is in } G_n(U)\},\\
G_n^4(U)=&\{\#[A\leftrightarrow B]:\hbox{both \#$A$  and \#$B$  or both \#[$\neg A$] and \#[$\neg B$] are in } G_n(U) \},\\
G_n^5(U)=&\{\#[\neg(A\vee B)]:\#[\neg A] \hbox{ and \#[$\neg B$] are in } G_n(U)\},\\
G_n^6(U)=&\{\#[\neg(A\wedge B)]:\#[\neg A] \hbox{ or \#$[\neg B]$ is in } G_n(U)\},\\
G_n^7(U)=&\{\#[\neg(A\rightarrow B)]:\#A \hbox{ and \#[$\neg B$] are in } G_n(U)\},\\
G_n^8(U)=&\{\#[\neg(A\leftrightarrow B)]:\hbox{both \#$A$ and \#[$\neg B]$ or both \#[$\neg A$] and \#$B$ are in } G_n(U) \},\\
G_n^9(U)=&\{\#[\neg(\neg A)]:\#A \hbox{ is in } G_n(U)\},
\end{aligned}
\end{equation}
and define
\begin{equation}\label{E204}
G_{n+1}(U)=G_n(U)\cup \bigcup_{k=1}^9 G_n^k(U).
\end{equation}
Because $\mathcal L$ is closed with respect to connectives $\neg$, $\vee$, $\wedge$, $\rightarrow$ and $\leftrightarrow$, it follows from the above construction that
$G_n(U)$ is defined for every $n\in\mathbb N_0$. 
Moreover, $G_n(U)\subseteq G_{n+1}(U)$ and  $G_n^k(U)\subseteq G_{n+1}^k(U)$ for all $n\in\mathbb N_0$ and $k=1,\dots,9$.
In particular, we can define
\begin{equation}\label{E21}
G(U)=\bigcup_{n=0}^\infty G_n(U).
\end{equation}
Because every set $G_n(U)$ is a subset of $D$, then also $G(U)$ is contained in $D$.


\section{Validity of rules (r1)--(r11)}\label{S3}

Now we are ready to prove our main result.

\begin{theorem}\label{T31} Let $U$ be a subset of $D$, and let the subsets $G(U)$ and $F(U)$ of $D$ be defined by \eqref{E21}, and by
\begin{equation}\label{E31}
F(U)=\{\#A: \#[\neg A]\in G(U)\}.
\end{equation}
Then rules (r1)--(r11) are valid.
\end{theorem}    

\begin{proof} Let $A$ be a sentence of $L$. It follows from the construction of $G(U)$ that \#$A$ is in $G(U)$ iff \#$A$ is in $W$, 
i.e., iff $A$ is true in the interpretation of $L$. Definition \eqref{E31} of $F(U)$ implies that \#$A$ is in $F(U)$ iff \#[$\neg A$]
is in $G(U)$ iff $\neg A$ is true in the interpretation of $L$. Because $L$ is an MA language, then  $\neg A$ is true in the interpretation of $L$ iff $A$ is false in the interpretation of $L$. This proves (r1).
\smallskip

Let $\mathbf n$ be a numeral. The construction of $G(U)$ implies that  \#$T(\mathbf n)$ is in $G(U)$ iff it is in $D_1(U)$, i.e., 
iff $\mathbf n=\left\lceil A\right\rceil$, where $A$ is a sentence of $\mathcal L$ and  \#$A$  is in $U$.
\#$T(\mathbf n)$ is by \eqref{E31} in $F(U)$ iff \#[$\neg T(\mathbf n)$] is in  $G(U)$ iff (by the construction of $G(U)$) 
\#[$\neg T(\mathbf n)$] is in $D_2(U)$ iff (by the definition of $D_2(U)$) $\mathbf n=\left\lceil A\right\rceil$, where $A$ is a sentence of $\mathcal L$ and  \#[$\neg A$]  is in $U$. This ends the proof of (r2).
\smallskip

In the proof of (r3) we need the following auxiliary result.
\begin{itemize}
\item[(r0)] If $A$ is a sentence of $\mathcal L$, then \#[$\neg(\neg A)$] is in $G(U)$ iff \#$A$ is in $G(U)$.
\end{itemize} 
Assume first that \#[$\neg(\neg A)$] is in $G_0(U)$. Then 
 \#[$\neg(\neg A)$] is in $W$, so that sentence $\neg(\neg A)$ is true in the interpretation of $L$. Since $L$ is an MA language, then  $A$ is true in the interpretation of $L$. Thus \#$A$ is in $W$, and hence in $G(U)$.

Assume next that the least of those $n$ for which  \#[$\neg(\neg A)$] is in $G_n(U)$ is $> 0$. The definition of $G_n(U)$ implies that if  \#[$\neg(\neg A)$] is in  $G_n(U)$, then  \#[$\neg(\neg A)$] is in  $G_{n-1}^9(U)$, so that \#$A$ is in $G_{n-1}(U)$, and hence in $G(U)$.

Thus \#$A$ is in $G(U)$ if \#[$\neg(\neg A)$] is in  $G_n(U)$ for some $n\in\mathbb N_0$, or equivalently, if \#[$\neg(\neg A)$] is in $G(U)$.

Conversely, if \#$A$ is in $G(U)$, then it is in $G_n(U)$ for some $n$, so that   \#[$\neg(\neg A)$] is in  $G_n^9(U)$, and hence
in $G_{n+1}(U)$, and thus in $G(U)$. This concludes the proof of (r0).
\smallskip

To prove (r3), let $A$ be a sentence of $\mathcal L$. It follows from \eqref{E31} that \#[$\neg A$] is in $G(U)$ iff \#$A$ is in $F(U)$.
Consider next the case when 
\#[$\neg A$] is in $F(U)$. By \eqref{E31} this holds iff  
\#[$\neg(\neg A)$] is in $G(U)$ iff (by (r0)) \#$A$ is in $G(U)$. 
 This ends the proof of (r3).
\smallskip


Let $A$ and $B$ be sentences of $\mathcal L$. If \#$A$ or \#$B$ is in $G(U)$, there is by \eqref{E21} an $n\in\mathbb N_0$ such that \#$A$ or \#$B$ is in $G_n(U)$. Thus  \#[$A\vee B$] is in $G_n^1(U)$, and hence in $G(U)$.

Conversely, assume that \#[$A\vee B$] is in $G(U)$. Then there is by \eqref{E21} an $n\in\mathbb N_0$ such that \#[$A\vee B$] is in $G_n(U)$.
Assume first that $n=0$. If \#[$A\vee B$] is in $G_0(U)$,  it is in $W$.
Thus  $A\vee B$ is true in the interpretation of $L$. Because $L$ is an MA language, then $A$ or $B$ is true in the interpretation of $L$, i.e., \#$A$ or \#$B$ is in $W$, and hence in $G(U)$.

Assume next that the least of those $n$ for which  \#[$A\vee B$] is in $G_n(U)$ is $> 0$. Then  \#[$A\vee B$] is in  $G_{n-1}^1(U)$, so that \#$A$ or \#$B$ is in $G_{n-1}(U)$, and hence in
$G(U)$.

Consequently, \#[$A\vee B$] is in $G(U)$ iff  \#$A$ or \#$B$ is in $G(U)$.

It follows from \eqref{E31} that

(a) \#[$A\vee B$] is  in $F(U)$ iff  \#[$\neg(A\vee B)$] is in $G(U)$.

If \#[$\neg(A\vee B)$] is in $G(U)$, there is by \eqref{E21} an $n\in\mathbb N_0$ such that \#[$\neg(A\vee B)$] is in $G_{n}(U)$.
Assume that $n=0$. If \#[$\neg(A\vee B)$] is in $G_0(U)$, it is in $W$.
Then  $\neg(A\vee B)$ is true in the interpretation of $L$, so that ($L$ is an MA language) $\neg A$ and $\neg B$ are true in the interpretation of $L$, i.e., \#[$\neg A$] and \#[$\neg B$] are in $W$, and hence in $G(U)$.

Assume next that the least of those $n$ for which  \#[$\neg(A\vee B)$] is in $G_n(U)$ is $> 0$. Then  \#[$\neg(A\vee B)$] is in  $G_{n-1}^5(U)$, so that \#[$\neg A$] and \#[$\neg B$] are in $G_{n-1}(U)$, and hence in
$G(U)$. Thus, \#[$\neg A$] and \#[$\neg B$] are in $G(U)$ if \#[$\neg(A\vee B)$] is in $G(U)$.

Conversely,  if \#[$\neg A$] and \#[$\neg B$] are in $G(U)$, there exist by \eqref{E21} $n_1,n_2\in\mathbb N_0$ such that \#[$\neg A$] is in $G_{n_1}(U)$ and \#[$\neg B$]  is in $G_{n_2}(U)$. Denoting $n=\max\{n_1,n_2\}$, then both \#[$\neg A$] and \#[$\neg B$] are in $G_n(U)$.
This result and the definition of $G_n^5(U)$ imply that \#[$\neg(A\vee B)$] is in $G_n^5(U)$, and hence in $G(U)$. Consequently,

(b)  \#[$\neg(A\vee B)$] is in $G(U)$ iff \#[$\neg A$] and \#[$\neg B$] are in $G(U)$ iff (by \eqref{E31}) \#$A$ and \#$B$ are in $F(U)$. 

Thus, by (a) and (b), \#[$A\vee B$] is  in $F(U)$ iff  \#$A$ and \#$B$ are in $F(U)$.
This concludes the proof of (r4).
\smallskip

The proofs for the validity of rules (r5)--(r7) are similar to that given above for rule (r4). 
\smallskip

\#[$\exists xT(x)$] is in $G(U)$ iff (by construction of $G(U)$) \#[$\exists xT(x)$] is in $G_0(U)$ iff $U$ is nonempty iff (by (r2)) 
\#$T(\mathbf n)$ is in $G(U)$ for some numeral $\mathbf n\in X$.

\#[$\exists xT(x)$] is in $F(U)$ iff (by (r3)) \#[$\neg \exists xT(x)$] is in $G(U)$ iff (by the construction of $G(U)$) \#[$\neg \exists xT(x)$] is in $G_0(U)$ iff $U=D$ iff (by (r2))  \#$T(\mathbf n)$ is in $F(U)$ for every numeral $\mathbf n\in X$. This concludes the proof of (r8). 
\smallskip

\#[$\forall xT(x)$] is in $G(U)$ iff (by construction of $G(U)$) \#[$\forall xT(x)$] is in $G_0(U)$ iff $U=D$ iff (by (r2)) \#$T(\mathbf n)$ is in $G(U)$ for every numeral $\mathbf n\in X$. 

\#[$\forall xT(x)$] is in $F(U)$ iff (by (r3)) \#[$\neg \forall xT(x)$] is in $G(U)$ iff (by the construction of $G(U)$) \#[$\neg\forall xT(x)$] is in $G_0(U)$ iff $D_2(U)$ is nonempty iff  \#$[\neg T(\mathbf n)]$ is in $G(U)$ at least for one numeral $\mathbf n\in X$ iff  \#$T(\mathbf n)$ is in $F(U)$ at least for one numeral $\mathbf n\in X$. This ends the proof of rule (r9).
\smallskip

Similar reasoning as in the above proofs of (r8) and (r9) can be used to verify that rules (r10) and (r11) are valid.   
\end{proof}


\section{Properties of $G(U)$ and $F(U)$ when $U$ is consistent}\label{S4}

In this section we shall prove some properties of $G(U)$ and $F(U)$, where $U$ is consistent. They are used in \cite{Hei14} to prove the existence of consistent fixed points of $G$.

\begin{lemma}\label{L01} Let $U$ be a consistent subset of $D$. Then $G(U)\cap F(U)=\emptyset$.
\end{lemma}

\begin{proof} If $A$ is in $L$, then  by rule (r1) \#$A$ is not in $G(U)\cap F(U)$ because $L$ is fully interpreted.

Let $\mathbf n$ be a numeral. \#$T(\mathbf n)$ is in $G_0(U)$ iff it is in $D_1(U)$ iff $\mathbf n= \left\lceil A\right\rceil$, where \#$A$ is in $U$. \#$T(\mathbf n)$ is  in $F(U)$ iff  \#[$\neg T(\mathbf n)]$ is in $G(U)$, or equivalently, in $D_2(U)$ iff $\mathbf n= \left\lceil A\right\rceil$, where
\#[$\neg A$] is in $U$. Thus \#$T(\mathbf n)$
cannot be both in $G(U)$ and in $F(U)$, and hence not in  $G_0(U)\cap F(U)$, because the consistency of $U$ implies that \#$A$ and \#[$\neg A$] cannot be both in $U$.

If $U$ is empty, then none of the G\"odel numbers
\#[$\exists xT(x)$],\#[$\neg\exists xT(x)$],
\#[$\forall xT(x)$],\#[$\neg\forall xT(x)$],

\#[$\forall x\neg T(x)$],\#[$\neg\forall x\neg T(x)$], 
\#[$\exists x\neg T(x)$], and \#[$\neg\exists x\neg T(x)$] is in $G_0(U)$.
Hence they are not in $G_0(U)\cap F(U)$. 

Assume next that $U$ is not empty. 
Because $U$ is consistent, it is a proper subset of $D$. Then rules (r1), (r8)-(r11), result (r0) and definitions \eqref{E201} and \eqref{E31} imply that 
not any of the above listed  G\"odel numbers is both in $G_0(U)$ and in $F(U)$, and hence 
 in $G_0(U)\cap F(U)$.

The above results and the definition of $G_0(U)$ imply that the
 induction hypothesis:
\begin{enumerate}
\item[(h0)] $G_n(U)\cap F(U)=\emptyset$ 
\end{enumerate}
holds for $n=0$.
If \#[$A\vee B$] is in $G_n^1(U)\cap F(U)$,  then \#$A$ or \#$B$ is in $G_n(U)$,  and both \#$A$ and \#$B$ are in $F(U)$ by (r4), so that \#$A$ or 
\#$B$ is in $G_n(U)\cap F(U)$.  Hence, if (h0) holds, then $G_n^1(U)\cap F(U)=\emptyset$.  

\#[$A\wedge B$] cannot be  in $G_n^2(U)\cap F(U)$, for otherwise both \#$A$ and \#$B$ are in $G_n(U)$, and at least one of \#$A$ and \#$B$ is in $F(U)$, so that \#$A$ or \#$B$ is in $G_n(U)\cap F(U)$, contradicting with (h0). Thus $G_n^2(U)\cap F(U)=\emptyset$ if (h0) holds.

  
If \#[$A\rightarrow B$] is in  $G_n^3(U)\cap F(U)$, then \#[$\neg A$] or \#$B$  is in $G_n(U)$ and both \#[$\neg A$] and \#$B$ are in $F(U)$.
But then \#[$\neg A$] or \#$B$  is in $G_n(U)\cap F(U)$. Thus  $G_n^3(U)\cap F(U)=\emptyset$ if (h0) holds.

If \#[$A\leftrightarrow B$] is in $G_n^4(U)\cap F(U)$, then both \#$A$ and \#$B$ or both \#[$\neg A$] and \#[$\neg B$] are in $G_n(U)$,
and both \#[$A$] and \#[$\neg B$] or both \#[$\neg A$] and \#[$B$] are in $F(U)$. Then one of G\"odel numbers \#$A$, \#$B$, \#[$\neg A$] and \#[$\neg B$] is in $G_n(U)\cap F(U)$, contradicting with (h0). Consequently,  $G_n^4(U)\cap F(U)=\emptyset$ if (h0) holds.

If \#[$\neg(A\vee B)$] is in $G_n^5(U)\cap F(U)$,  then \#[$\neg A$] and \#[$\neg B$] are in $G_n(U)$, and \#[$A\vee B$] is in $G(U)$, i.e.,
\#$A$ or \#$B$ is in $G(U)$, or equivalently, \#[$\neg A$] or \#[$\neg B$] is in $F(U)$. Thus \#[$\neg A$] or \#[$\neg B$] is in $G_n(U)\cap F(U)$.  Hence, if (h0) holds, then $G_n^5(U)\cap F(U)=\emptyset$.  

If \#[$\neg(A\wedge B)$] is in $G_n^6(U)\cap F(U)$,  then \#[$\neg A$] or \#[$\neg B$] is in $G_n(U)$, and \#[$A\wedge B$] is in $G(U)$, or equivalently, \#$A$ and \#$B$ are in $G(U)$, i.e., \#[$\neg A$] and \#[$\neg B$] are in $F(U)$.  Consequently, \#[$\neg A$] or \#[$\neg B$] is in $G_n(U)\cap F(U)$, contradicting with (h0). Thus $G_n^6(U)\cap F(U)=\emptyset$
if (h0) holds.  

If \#[$\neg(A\rightarrow B)$] is in $G_n^7(U)\cap F(U)$,  then \#$A$ and \#[$\neg B$] are in $G_n(U)$, and \#[$A\rightarrow B$] is in $G(U)$, i.e.,
\#[$\neg A$] or \#$B$ is in $G(U)$, or equivalently, \#$A$ or \#[$\neg B$] is in $F(U)$. Thus \#$A$ or \#[$\neg B$] is in $G_n(U)\cap F(U)$. Hence, if (h0) holds, then $G_n^7(U)\cap F(U)=\emptyset$.  

\#[$\neg(A\leftrightarrow B)$] cannot be  in $G_n^8(U)\cap F(U)$, for otherwise both \#[$A$] and \#[$\neg B$] or both \#[$\neg A$] and \#[$B$] are in $G_n(U)$, and \#[$A\leftrightarrow B$] is in $G(U)$, i.e.,
both \#[$\neg A$] and \#[$\neg B$]  or both \#$A$ and \#$B$ are in $F(U)$. Thus one of G\"odel numbers \#$A$, \#$B$, \#[$\neg A$] and \#[$\neg B$] is in $G_n(U)\cap F(U)$, contradicting with (h0). Thus (h0) implies that $G_n^8(U)\cap F(U)=\emptyset$.

If \#[$\neg(\neg A)$] would be  in $G_n^9(U)\cap F(U)$, then \#$A$ would be in $G_n(U)$ and \#[$\neg(\neg A)$], or equivalently, by (r0), \#$A$ would be in $F(U)$, so that  \#$A$ would be in $G_n(U)\cap F(U)$. Consequently, $G_n^9(U)\cap F(U)=\emptyset$
if (h0) holds.  

Because $G_{n+1}(U)=G_n(U)\cup \bigcup_{k=1}^9 G_n^k(U)$, the above results imply that $G_{n+1}(U)\cap F(U)=\emptyset$ if (h0) holds.
Since it holds when $n=0$, the above proof shows by induction
that it holds for all $n\in\mathbb N_0$. 

If \#$A$ is in $G(U)$, it is by \eqref{E21} in $G_n(U)$ for some $n\in\mathbb N_0$. Because (h0) holds, then \#$A$ is not in $F(U)$. Consequently, $G(U)\cap F(U)=\emptyset$.
\end{proof}

\begin{lemma}\label{L201} If $U$ is a consistent subset of $D$, then both $G(U)$ and $F(U)$ are consistent.
\end{lemma}

\begin{proof} If $G(U)$ is not consistent, then there is such a sentence $A$ of $\mathcal L$, that \#$A$ and \#[$\neg A$] are in $G(U)$. Because \#[$\neg A$] is in $G(U)$, then \#$A$ is also in $F(U)$ by \eqref{E31}, and hence in $G(U)\cap F(U)$. But then, by Lemma \ref{L01}, $U$ is not consistent.
Consequently, if $U$ is  consistent, then $G(U)$ is consistent. The proof that $F(U)$ is consistent if $U$ is, is similar.
\end{proof}

\begin{lemma}\label{L203} Assume that $U$ and $V$ are consistent  subsets of  $D$, and that $U\subseteq V$.
 Then $G(U)\subseteq G(V)$ and $F(U)\subseteq F(V)$.  
\end{lemma}

\begin{proof} Assume that $U$ and $V$ are consistent subsets of $D$, and that $U\subseteq V$. Let $A$ be a sentence of $L$. By rule (r1) \#$A$ is in $G(U)$ and also in $G(V)$ iff \#$A$ is in $W$.
  
Let $\mathbf n$ be a numeral. If \#$T(\mathbf n)$ is in $G_0(U)$,  then   $\mathbf n =\left\lceil A\right\rceil$, where \#$A$ is in $U$. Because $U\subseteq V$, then \#$A$ is also in $V$, whence \#$T(\mathbf n)$ is in $G_0(V)$.
If \#[$\neg T(\mathbf n)]$ is in $G_0(U)$,  then  $\mathbf n =\left\lceil A\right\rceil$, where \#[$\neg A$] is in $U$. Because $U\subseteq V$, then \#[$\neg A$] is also in $V$, whence \#[$\neg T(\mathbf n)$] is in $G_0(V)$.

If \#[$\exists xT(x)$] is  in $G_0(U)$, then $U$ is nonempty. Because $U\subseteq V$, then also $V$ is nonempty, whence \#[$\exists xT(x)]$ is in $G_0(V)$. Consequently,   
\#[$\exists xT(x)$] is in $G_0(V)$ whenever it is in $G_0(U)$. The similar reasoning shows that \#[$\neg(\forall x\neg T(x))$] is in $G_0(V)$ whenever it is in $G_0(U)$.

\#[$\exists x\neg T(x)$] is  in $G_0(U)$ iff $D_2(U)$ is nonempty, i.e., such a sentence $A$ exists in $\mathcal L$ that \#[$\neg A]$ is in $U$.
But then \#[$\neg A]$ is also in $V$, i.e.,  $D_2(V)$ is nonempty, whence \#[$\exists x\neg T(x)$] is  in $G_0(V)$.
Similarly it can be shown that \#[$\neg\forall xT(x)$] is in $G_0(V)$ whenever it in $G_0(U)$.

As consistent sets both $U$ and $V$ are proper subsets of $D$. Thus the G\"odel numbers \#[$\forall xT(x)$], \#[$\neg(\exists x\neg T(x))$,\#[$\neg\exists xT(x)$] and \#[$\forall x\neg T(x)$] are  neither in  $G_0(U)$ nor in $G_0(V)$. 

The above results imply that $G_0(U)\subseteq G_0(V)$. 
Make an induction hypothesis:
\begin{enumerate}
\item[(h1)] $G_n(U)\subseteq G_n(V)$.
\end{enumerate}
The definitions of the sets $G_n^k(U)$, $k=1,\dots,9$, given in \eqref{E203}, together with (h1), imply that
$G_n^k(U)\subseteq G_n^k(V)$ for each $k=1,\dots,9$. Thus
$$
G_{n+1}(U)=G_n(U)\cup \bigcup_{k=1}^9 G_n^k(U)\subseteq G_n(V)\cup \bigcup_{k=1}^9 G_n^k(V)=G_{n+1}(V).
$$
Because (h1) is shown to hold when $n=0$, then it holds for every $n\in\mathbb N_0$.

If \#$A$ is in $G(U)$, it is by \eqref{E21} in $G_n(U)$ for some $n$. Thus   \#$A$ is in $G_n(V)$ by (h1), and hence in $G(V)$.
Consequently, $G(U)\subseteq G(V)$. 

If \#$A$ is in $F(U)$, it follows from \eqref{E31} that \#[$\neg A$] is in $G(U)$. Because $G(U)\subseteq G(V)$, then \#[$\neg A$] is in $G(V)$
This implies  by \eqref{E31} that  \#$A$ is in $F(V)$. Thus $F(U)\subseteq F(V)$.
\end{proof}

Let $\mathcal P$ denote the family of all consistent subsets of the set $D$ of 
G\"odel numbers of sentences of $\mathcal L$. 
According to Lemma \ref{L201} the mapping $G:=U\mapsto G(U)$ maps $\mathcal P$ into $\mathcal P$.
Assuming  that $\mathcal P$ is ordered by inclusion, it follows from Lemma \ref{L203} that $G$ is order preserving.

In the formulation and the proof of Theorem \ref{T1} below  transfinite sequences of $\mathcal P$ indexed by von Neumann ordinals are used. Such a sequence $(U_\lambda)_{\lambda\in\alpha}$ of $\mathcal P$ is said to be increasing if 
$U_\mu\subseteq U_\nu$ whenever $\mu\in\nu\in\alpha$, and strictly increasing if 
$U_\mu\subset U_\nu$ whenever $\mu\in\nu\in\alpha$.

\begin{lemma}\label{L204} Assume that $(U_\lambda)_{\lambda\in\alpha}$ a strictly increasing sequence of consistent subsets of $D$.  
Then  

(a) $(G(U_\lambda))_{\lambda\in\alpha}$  is an increasing sequence of consistent subsets of $D$.

(b) The set $U_\alpha = \underset{\lambda\in\alpha}{\bigcup}G(U_\lambda)$ is consistent.
\end{lemma}

\begin{proof} (a) Consistency of the sets $G(U_\lambda)$, $\lambda\in\alpha$, follows from Lemma \ref{L201} because the sets  $U_\lambda$, $\lambda\in\alpha$, are consistent. 

Because $U_\mu\subset U_\nu$ whenever $\mu\in\nu\in\alpha$, then $G(U_\mu)\subseteq G(U_\nu)$ whenever $\mu\in\nu\in\alpha$, by Lemma \ref{L203}, whence the sequence  $(G(U_\lambda))_{\lambda\in\alpha}$  is increasing. This proves (a).
\smallskip

(b) To prove that the set $\underset{\lambda\in\alpha}{\bigcup}G(U_\lambda)$  is consistent, assume on the contrary that there exists such a sentence $A$ in $\mathcal L$ that  both \#$A$ and \#[$\neg A$] are in $\underset{\lambda\in\alpha}{\bigcup}G(U_\lambda)$.
Thus there exist $\mu,\,\nu\in\alpha$ such that \#$A$ is in $G(U_\mu)$ and \#[$\neg A$] is in $G(U_\nu)$. Because
$G(U_\mu)\subseteq G(U_\nu)$ or $G(U_\nu)\subseteq G(U_\mu)$, then both  \#$A$ and \#[$\neg A$] are in $G(U_\mu)$ or in $G(U_\nu)$.
But this is impossible, since both $G(U_\mu)$ and $G(U_\nu)$ are consistent by (a). Thus, the set $\underset{\lambda\in\alpha}{\bigcup}G(U_\lambda)$  is consistent, so that the conclusion of (b) holds. 
\end{proof}

A subset $V$ of $D$ is called  sound if $V\subseteq G(V)$. Let $V$ be a subset of the set $W$ of G\"odel numbers of all those sentences of $L$ which are true in the interpretation of $L$. Since $L$ is an MA language, then $W$ is consistent. Thus also $V$ is consistent. Because $W=G_0(\emptyset)\subset G(\emptyset)$,
then $V\subset G(\emptyset)\subseteq G(V)$. Thus $V$ is also sound.
  
The following fixed point theorem is an application of Lemmas \ref{L201}, \ref{L203} and \ref{L204}, and is proved in \cite{Hei14}.

\begin{theorem}\label{T1} 
If $V\in \mathcal P$ is  sound, there exists the smallest of those  fixed points of $G$  which contain $V$. This fixed point is the last member of the union of those transfinite sequences 
$(U_\lambda)_{\lambda\in\alpha}$ of $\mathcal P$ which satisfy
\begin{itemize}
\item[(C)] $(U_\lambda)_{\lambda\in\alpha}$ is strictly increasing,  
$U_0=V$, and if $0\in\mu\in \alpha$, then
$U_\mu = \underset{\lambda\in\mu}{\bigcup}G(U_\lambda)$.  
\end{itemize}  
\end{theorem}

The union $(U_\lambda)_{\lambda\in\gamma}$ of the transfinite sequences satisfying (C) can be characterized as follows (cf. \cite{Hei99}).
 \begin{itemize}
\item[(I)] $U_0=V$. If $\lambda$ is in $\gamma$, then $\lambda+1$ is in $\gamma$ iff $U_\lambda\subset G(U_\lambda)$, in which case $U_{\lambda+1}=G(U_\lambda)$. \\ If $\alpha$ is a limit ordinal, and $\lambda$ is in $\gamma$ for each $\lambda\in\alpha$, then $\alpha$ is in $\gamma$, and
$U_\alpha=\underset{\lambda\in\alpha}{\bigcup}U_\lambda$. 
\end{itemize}

It follows from (I) that the sequence $(U_\lambda)_{\lambda\in\gamma}$ begins with sets $U_0=V$, $U_{n+1}=G(U_n)$,
$n=0,1,\dots$, $U_\omega= \underset{n\in\omega}{\bigcup}U_n$, $U_{\omega+n+1}=G(U_{\omega+n})$, $n=0,1,\dots$, e.t.c., as long as the so defined sets exist and contain strictly previous sets. Because $(U_\lambda)_{\lambda\in\gamma}$ is a strictly increasing sequence of subsets of a countable set $D$, then $\gamma$ is a countable ordinal. In this sense the smallest fixed point of $G$ that contains $V$ is determined by a countable recursion method. By \cite[Corollary 3.1]{Hei15} this fixed point $U$ is the smallest of all fixed points of $G$ if $V$ is a subset of $W$.  Those sentences of $\mathcal L$ whose G\"odel numbers are in $U$, or equivalently, in $G(U)$, and their negations, whose G\"odel numbers belong by (r3) to $F(U)$, form an MA language $\mathcal L_0$ which 
contains its truth predicate.

\end{document}